\documentclass[12pt, reqno]{amsart}

\usepackage{amsmath}
\usepackage{amsthm}
\usepackage{amssymb}
\usepackage{amsrefs}
\usepackage{mathrsfs}
\usepackage[usenames]{color}
\usepackage{ifthen}
\usepackage{testhyphens}
\newtheorem{thm}{}[section]
\newtheorem{theorem}[thm]{Theorem}

\newtheorem{lemma}[thm]{Lemma}
\newtheorem{proposition}[thm]{Proposition}

\theoremstyle{definition}
\newtheorem{definition}[thm]{Definition}
\theoremstyle{remark}
\newtheorem{remark}[thm]{Remark}

\numberwithin{equation}{section}
\allowdisplaybreaks

\newcommand{\Env}[2][]{%
\ifthenelse{ \equal{#1}{} }
{\ensuremath{#2_{\mathsf{c}}}}
{\ensuremath{#2_{\mathsf{c},#1}}}
}

\newcommand{\UU}{\ensuremath{\mathbb{U}}}

\newcommand{\uu}{\ensuremath{\mathbf{u}}}
\newcommand{\vv}{\ensuremath{\mathbf{v}}}
\newcommand{\LL}{\ensuremath{\mathbb{L}}}
\newcommand{\NN}{\ensuremath{\mathbb{N}}}
\newcommand{\FF}{\ensuremath{\mathbb{F}}}
\newcommand{\XX}{\ensuremath{\mathbb{X}}}
\newcommand{\YY}{\ensuremath{\mathbb{Y}}}
\newcommand{\xx}{\ensuremath{\mathbf{x}}}
\newcommand{\yy}{\ensuremath{\mathbf{y}}}
\newcommand{\ee}{\ensuremath{\mathbf{e}}}
\newcommand{\Klt}{\ensuremath{\mathcal{K}}}
\newcommand{\Cvx}{\ensuremath{\mathcal{C}}}
\newcommand{\EE}{\ensuremath{\mathcal{E}}}
\newcommand{\BB}{\ensuremath{\mathcal{B}}}
\newcommand{\Id}{\ensuremath{\mathrm{Id}}}
\newcommand{\supp}{\operatorname{supp}}

\begin{document}

\title[]{Projections and unconditional bases in direct sums of $\ell_p$ spaces, $0<p\le \infty$.}

\author[F. Albiac]{F. Albiac}
\address{Department of Mathematics, Statistics and Computer Sciences, and InaMat\\ Universidad P\'ublica de Navarra\\
Pamplona 31006\\ Spain} 
\email{fernando.albiac@unavarra.es}

\author[J. L. Ansorena]{J. L. Ansorena}
\address{Department of Mathematics and Computer Sciences\\
Universidad de La Rioja\\
Logro\~no 26004\\ Spain} 
\email{joseluis.ansorena@unirioja.es}



\begin{abstract}
We show that every unconditional basis in a \emph{finite} direct sum $\bigoplus_{p\in A} \ell_p$, with $A\subset (0,\infty]$, splits into unconditional bases of each summand. This settles a 40 year old question raised in [A. Orty\'nski, \emph{Unconditional bases in $\ell_{p}\oplus\ell_{q},$ $0<p<q<1$}, Math. Nachr. \textbf{103} (1981), 109-116]. As an application we obtain that for any $A\subset (0,1]$ finite, the spaces $Z=\bigoplus_{p\in A} \ell_p$, $Z\oplus \ell_{2}$, and $Z\oplus c_{0}$ have a unique unconditional basis up to permutation.
\end{abstract}

\thanks{Both authors supported by the Spanish Ministry for Science, Innovation, and Universities, Grant PGC2018-095366-B-I00 for \emph{An\'alisis Vectorial, Multilineal y Approximaci\'on}. The first-named author also acknowledges the support from Spanish Ministry for Economy and Competitivity, Grant MTM2016-76808-P for \emph{Operators, lattices, and structure of Banach spaces}.}

\maketitle

\section{Introduction}
\noindent From a structural point of view, an important topic in vintage Banach space theory is the description of the unconditional bases in a given space. This note is devoted to advance the state of art of this subject by proving the following result.

\begin{theorem}[Main Theorem]\label{thm:main} Let $A$ be a finite set of indexes contained in $(0,\infty]$. If $(\xx_j)_{j\in J}$ is an unconditional basis of $\bigoplus_{p\in A} \ell_p$ (with the convention that $\ell_\infty$ means $c_0$ if $\infty\in A$) then there is a partition $(J_p)_{p\in A}$ of $J$ in such a way that $(\xx_j)_{j\in J_p}$ is a basis of $\ell_p$ for each $p\in A$.
\end{theorem}

Theorem~\ref{thm:main} builds on previous work on the subject initiated by Edelstein and Wojtaszczyk in 1976, and continued by Orty\'nski in 1981. In \cite{EdWo1976} the authors described unconditional bases in finite direct sums of $\ell_p$ spaces for values of $p$ in the locally convex range, i.e., when the set $A$ in the statement of the theorem is contained in $[1,\infty]$. Five years later, Ort\'ynski ventured himself across the nonlocally convex border and extended Edelstein and Wojtaszczyk's result to the case when $A\subseteq(0,1)\cup (1,\infty]$. However, Orty\'nski's methods break down precisely when the set $A$ contains the index $p=1$, which led him to raise the question whether his theorem remained valid without any restriction on $p$ (see Remark 2 in \cite{Ortynski1981}*{p.\ 115}). Our goal is to fill this gap in the literature. With the help of fresh techniques derived from the theory of anti-Euclidean spaces developed by Casazza and Kalton in the late 1990' (hence not yet available when \cite{Ortynski1981} was written) we are able to prove Theorem~\ref{thm:main} in its whole generality, thus providing a positive answer to Orty\'nski's question.

\section{Preliminaries}
\noindent Throughout this article we employ the notation commonly used in Banach space theory as can be found in \cite{AlbiacKalton2016}. The only exception is that, for convenience in writing our results in a unified way, unless otherwise stated the symbol $\ell_{\infty}$ denotes the space of sequences tending to zero equipped with the supremum norm, usually denoted by $c_{0}$. Next we single out the notation that it is more heavily used. We denote by $\FF$ the real or complex scalar field and by $c_{00}(J)$ the set of all $(a_j)_{j\in J}\in \FF^J$ such that $|\{j\in J \colon a_j\not=0\}|<\infty$. We write by $\EE_{J}:=(\ee_j)_{i\in J}$ for the canonical unit vector system of $\FF^J$, i.e., $\ee_j=(\delta_{j,i})_{i\in J}$ for each $j\in J$, where $\delta_{j,i}=1$ if $i=j$ and $\delta_{i,j}=0$ otherwise. The symbol $\alpha_i\lesssim \beta_i$ for $i\in I$ means that the families of positive real numbers $(\alpha_i)_{i\in I}$ and $(\beta_i)_{i\in I}$ verify $\sup_{i\in I}\alpha_i/\beta_i <\infty$. If $\alpha_i\lesssim \beta_i$ and $\beta_i\lesssim \alpha_i$ for $i\in I$ we say $(\alpha_i)_{i\in I}$ are $(\beta_i)_{i\in I}$ are equivalent, and we write $\alpha_i\approx \beta_i$ for $i\in I$. The closed linear span of a set of vectors $A$ in a quasi-Banach space $\XX$ is denoted by $[A]$. If we have a finite family $(\XX_i)_{i\in F}$ of quasi-Banach spaces, the Cartesian product $\bigoplus_{i\in F}\XX_i$ of $\XX_{i}$'s with coordinatewise algebraic operations is a quasi-Banach space with the quasi-norm
\[
\left\Vert (\xx_i)_{i}\right\Vert=\sup_{i\in F} \Vert \xx_i\Vert,\quad \xx_i\in\XX_i.
\]
If $F=\{i, \dots, N\}$ and $\XX_i=\XX$ for $i\in F$ we simply write $\XX^N$ and call the resulting direct sum the $N$-fold product of $\XX$. A $2$-fold product will be called a \emph{square}.

For expositional ease and to fix the notation, below we gather well-known concepts scattered through the literature and the basic facts about them that we will need later on, both in Sections~\ref{Prepa} and \ref{MainSec}.

\subsection*{Quasi-Banach spaces and envelopes}
Quasi-Banach spaces provide the natural framework for our work. Recall that a \emph{quasi-Banach} space $\XX$ is a locally bounded topological vector space. This is equivalent to saying that the topology on $\XX$ is induced by a \emph{quasi-norm}, i.e., a map $\Vert \cdot\Vert\colon \XX\to [0,\infty)$ with the following properties:
\begin{enumerate}
\item[(i)] $\Vert x\Vert = 0$ if and only if $x=0$;

\item[(ii)] $\Vert \alpha x\Vert =|\alpha| \Vert x\Vert$ for $\alpha\in \FF$ and, $x\in \XX$,

\item[(iii)] there is a constant $\kappa\ge 1$ so that for all $x$ and $y$ in $\XX$ we have
\[
\Vert x +y\Vert \le \kappa (\Vert x\Vert +\Vert y\Vert).
\]
\end{enumerate}
If it is possible to take $\kappa=1$ we obtain a norm. A quasi-norm $\Vert\cdot\Vert$ is called \emph{$p$-norm} $(0<p\le 1)$ if it is $p$-subadditive, that is,
\[
\Vert x+y\Vert^{p}\le \Vert x\Vert^{p} +\Vert y\Vert ^{p},\qquad \forall x, y\in \XX.
\]
In this case the unit ball of $\XX$ is an absolutely $p$-convex set and $\XX$ is said to be $p$-convex. A quasi-norm clearly defines a metrizable vector topology on $\XX$ whose base of neighborhoods of zero is given by sets of the form $\{x\in \XX \colon \Vert x\Vert<1/n\}$, $n\in \NN$. If such topology is complete we say that $(\XX, \Vert\cdot\Vert)$ is a quasi-Banach space. If the quasi-norm is $p$-subadditive for some $0<p\le 1$, $\XX$ is said to be a $p$-convex quasi-Banach space ($p$-Banach space for short). Let us notice that a $p$-subadditive quasi-norm $\Vert\cdot\Vert$ induces an invariant metric topology on $\XX$ by the formula $d(x,y)= \Vert x-y\Vert^{p}$.

All Banach spaces are locally convex thanks to the triangle law of the norm, which translates geometrically in the spaces having a convex set as unit ball. In contrast, a quasi-Banach space $(\XX,\Vert\cdot\Vert)$ need not be $p$-convex for any $p\le 1$. The Aoki-Rolewicz theorem \cites{Aoki1942,Rolewicz1957} guarantees that, at least, $\XX$ is $p$-normable for some $0<p\le1$, i.e., $\XX$ can be endowed with an equivalent quasi-norm which is $p$-subadditive.

When dealing with a quasi-Banach space $\XX$ it is often convenient to know which is, the ``smallest'' Banach space containing $\XX$. Formally, the \emph{Banach envelope} of a quasi-Banach space $\XX$ consists of a Banach space $\widehat{\XX}$ together with a linear contraction $J_\XX\colon\XX \to \widehat{\XX}$ satisfying the following universal property: for every Banach space $\YY$ and every linear contraction $T\colon\XX \to\YY$ there is a unique linear contraction $\widehat{T}\colon \widehat{\XX}\to \YY$ such that $\widehat{T}\circ J_\XX=T$. Given a family of vectors (usually a basis) $\BB$ in $\XX$ we put $\widehat{\BB}:=J_\XX(\BB)$. A map $J$ from $\XX$ into a Banach space $\YY$ is said to be an \emph{envelope map} if the associated map $\widehat{J}\colon\widehat{\XX}\to\YY$ is an isomorphic embedding. If $J\colon\XX\to\YY$ is an envelope map there is a natural identification of $(J_\XX,\widehat{\XX})$ with $(J,\overline{J(\XX)})$. For instance, the Banach envelope of $\ell_p$ for $p<1$ is (naturally identified with) $\ell_{1}$, whereas the Banach envelope of $L_{p}([0,1])$ is $\{0\}$. The unfamiliar reader will find general information about quasi-Banach spaces in \cite{KPR1984}.

\subsection*{Unconditionality of bases and basic sequences}
Let $\XX$ be a quasi-Banach space and $J$ be a (finite or infinite) countable set. Given a sequence $(f_j)_{j\in J}$, a series $\sum_{j\in J} f_j$ converges unconditionally to $f\in \XX$, and write $f= \sum_{j\in J} f_j$ unconditionally, if either $J$ is finite and $f=\sum_{j\in J} f_j$ or $J$ is infinite and $f=\sum_{n=1}^\infty f_{\pi(n)}$ for every bijection $\pi\colon\NN\to J$.

A sequence $\BB=(\xx_j)_{j\in J}$ in $\XX$ is an \emph{unconditional basic sequence} if for every $f\in[\xx_j \colon j\in J]$ there is a unique family $(a_j)_{j\in J}$ in $\FF$ such that the series $\sum_{j\in J} a_j \, \xx_j$ converges unconditionally to $f$. Given $1\le C<\infty$ we say that $\BB$ is a $C$-unconditional basic sequence if $\xx_j\not=0$ for every $j\in J$ and there is a constant $C$ such that
\begin{equation*}
\left\Vert \sum_{j\in J} a_j \, \xx_j \right\Vert \le C \left\Vert \sum_{j\in J} b_j \, \xx_j \right\Vert
\end{equation*}
for every $(a_j)_{j\in J}$ and $(b_j)_{j\in J}$ in $c_{00}(J)$ such that $|a_j|\le |b_j|$ for all $j\in J$. It is well-known that $\BB$ is an unconditional basic sequence if and only it is a $C$-unconditional basic sequence for some $C$. The smallest constant $C\ge 1$ for which an unconditional basic sequence is $C$-unconditional will be called the \emph{unconditional constant} of $\BB$. Of course, a finite unconditional basis is nothing but a finite linearly independent family. All unconditional basic sequences are assumed to be infinite, unless otherwise stated. It is customary to index basic sequences with the set $\NN$ of natural numbers, and to index finite basic sequences with subsets $\{1,\dots,m\}$ of $\NN$. We prefer to use an arbitrary countable set $J$ to emphasize that, when dealing with unconditional basic sequences, it is unnecessary to arrange the sequences in a particular way.

The following lemmas just follow from the universal property of Banach envelopes (see, e.g., \cite{AABW2019}*{Proposition 9.9}). We write them down for further reference.

\begin{lemma}\label{lem:envelope} If $\BB$ is a semi-normalized complemented unconditional basic sequence of a quasi-Banach space $\XX$ then $\widehat{\BB}$ is a semi-normalized complemented unconditional basic sequence of $\widehat{\XX}$.
\end{lemma}

\begin{lemma}\label{lem:envelopebis} If $\YY$ is a complemented subspace of a quasi-Banach space $\XX$, then $J_\XX|_\YY \colon \YY \to \widehat{\XX}$ is an envelope map.
\end{lemma}

\begin{lemma}\label{lem:envelopetris} If $\XX_1$ and $\XX_2$ are quasi-Banach spaces, then
\[
(J_{\XX_1}, J_{\XX_2}) \colon \XX_1\oplus\XX_2 \to \widehat{\XX_1}\oplus\widehat{\XX_2}
\]
is an envelope map.
\end{lemma}

An \emph{unconditional basis} of a quasi-Banach space $\XX$ is an unconditional basic sequence whose closed linear span is $\XX$. Given a basis $\BB=(\xx_j)_{j\in J}$ of $\XX$ there is unique sequence $\BB^*=(\xx_j^*)_{j\in J}$ in $\XX^*$ such that $f=\sum_{j\in J} \xx_j^*(f) \, \xx_j$ for every $f\in\XX$. The \emph{support} of $f=\sum_{j\in J}a_j\, \xx_j\in\XX$ with respect to $\BB$ is the set $\supp(f)=\{j\in J\colon a_j\not=0\}$, and the support of $f^*\in\XX^*$ with respect to $\BB$ is the set $\supp(f^*)=\{j\in J\colon f^*(\xx_j)\not=0\}$. Associated to an unconditional basis $\BB$ of $\XX$, for each $\gamma=(\gamma_j)_{j\in J}\in \ell_\infty(J)$ there is a bounded linear map $S_\gamma\colon\XX\to\XX$ given by
\[
S_\gamma(f)=\sum_{j\in J} \gamma_j\,\xx_j^*(f) \, \xx_j, \quad f\in \XX.
\]
In case that $\gamma$ is the indicator function of a set $A\subseteq J$ we put $S_A:=S_\gamma$ and say that the operator $S_A$ is the \textit{coordinate projection} on $A$.

Suppose that $F$ is a finite set of indices and that $\BB_i=(\xx_{i,j})_{j\in J_i}$ is an unconditional basis of a quasi-Banach space $\XX_i$ for $i\in F$. Then the sequence
\[\bigoplus_{i\in F} \BB_i :=(\xx_{i,j})_{(i,j)\in \cup_{i\in F} \{i\} \times J_i }\subset \bigoplus_{i\in F} \XX_i\] defined by
\[
\xx_{i,j} =(\xx_{i,j,k})_{k\in F}, \text{ where } \xx_{i,j,k}=\begin{cases} \xx_j& \text{ if }k=i, \\ 0 & \text{ otherwise,} \end{cases}
\]
is an unconditional basis of $\bigoplus_{i\in F} \XX_i$. If $F=\{1,\dots,N\}$ and $\XX_i=\XX$ and $\BB_i=\BB$ for all $i\in F$, we put $\BB^N=\bigoplus_{i\in F} \BB_i$ and say that $\BB^N$ is the $N$-fold product of $\BB$.

A \emph{subbasis} of a basis $\BB=(\xx_j)_{j\in J}$ of a quasi-Banach space $\XX$ is a sequence of the form $(\xx_j)_{j\in I}$ for some $I\subseteq J$. A \emph{block basic sequence} with respect to $\BB$ is a sequence $(\uu_i)_{i\in G}$ in $\XX$ such that $(\supp(\uu_i))_{i\in G}$ is a family of pairwise disjoint finite sets. A block basic sequence of an unconditional basis is an unconditional basic sequence, and the subbases of $\BB$ are in particular block basic sequences.

\subsection*{Domination and equivalence of sequences}
Let $\BB_0=(\xx_j)_{j\in J}$ and $\BB_1=(\yy_i)_{i\in I}$ be two sequences in quasi-Banach spaces $\XX$ and $\YY$ respectively. We say that $\BB_0$ \emph{dominates a permutation} of $\BB_1$ if there are a bounded linear map $T\colon[\BB_0]\to \BB_1]$ and a bijection $\pi\colon J\to I$ such that $T(\xx_j)=\yy_{\pi(j)}$ for all $j\in J$. In the case when $\pi$ is the identity map we say that $\BB_0$ dominates $\BB_1$. If $T$ is an isomorphism from $[\BB_0]$ onto $[\BB_1]$ we say $\BB_0$ and $\BB_1$ are \emph{permutatively equivalent} (\emph{equivalent} if, moreover, $\pi$ is the identity map), and we write $\BB_0\sim\BB_1$. Note that if $\BB_0$ and $\BB_1$ are permutatively equivalent and $\BB_0$ is an unconditional basic sequence, then $\BB_1$ is an unconditional basic sequence.

An unconditional basis $\BB = (\xx_j)_{j=1}^\infty$ of a quasi-Banach space $\XX$ is said to be \emph{subsymmetric} if it is equivalent to $(\xx_{\phi(j)})_{j=1}^\infty$ for every increasing map $\phi\colon\NN\to\NN$. For example, the canonical basis of $\ell_{p}$, which we denote by $\EE_{\ell_{p}}$ from now on, is subsymmetric (in fact it is \emph{symmetric}, see \cites{KaPe1961,Singer1962}).

Let us record a very simple, yet useful, property.

\begin{lemma}\label{lem:sumSSK3} Suppose that $|F|<\infty$ and that for each $i\in F$, $\BB_i$ is a subsymmetric basis in a quasi-Banach space $\XX_{i}$. Then every subbasis of $\bigoplus_{i\in F} \BB_i$ is permutatively equivalent to $\bigoplus_{i\in G} \BB_i$ for some $G\subseteq F$.
\end{lemma}

\begin{proof}Let $\BB_0$ be a subbasis of $\bigoplus_{i\in I} \BB_i$ and write $\BB_0=\bigoplus_{i\in F} \BB_{i,0}$, where $\BB_{i,0}$ is a (finite or infinite) subbasis of $\BB_i$, $i\in F$.
The set
\[
G=\{ i \in F \colon \BB_{i,0} \text{ is infinite}\}
\]
is clearly nonempty, and $\BB_0\sim \BB' \oplus ( \bigoplus_{i\in G} \BB_{i})$ for some finite basis $\BB'$. Pick $g\in G$. Dropping a finite number of elements from $\BB_g$, we infer that $\BB_g\sim \BB'\oplus \BB_g$. Therefore, $\BB_0\sim \bigoplus_{i\in G} \BB_{i}$.
\end{proof}

\subsection*{Lattice structure and convexity conditions}
We are interested in quasi-Banach lattices of functions defined on a countable set $J$, i.e., a quasi-Banach space $\LL\subseteq \FF^J$ such that $\ee_j=(\delta_{i,j})_{i\in J}\in \LL$ for every $j\in J$, and whenever $f\in \LL$ and $g\in\FF^J$ satisfy $|g|\le |f|$ we have $g\in\LL$ and $\Vert g\Vert \le \Vert f \Vert$. In this case, the unit vector system $\EE_J=(\ee_j)_{j\in J}$ is a $1$-unconditional basic sequence for $\LL$. If, additionally, $\EE_J$ is normalized we will say that $\LL$ is a \emph{sequence space}. Recall that a quasi-Banach lattice $\LL$ is said to be \emph{$p$-convex} (resp., \emph{$q$-concave}), where $0<p\le \infty$ (resp. $0< q\le \infty$) if there is a constant $M>0$ such that for any $x_{1},\dots,x_{m}\in\LL$ and $m\in{\NN}$ we have
\begin{equation}\label{latticedef}
\left\Vert\left(\sum_{n=1}^{m}\vert x_{n}\vert^{p}\right)^{1/p}\right\Vert \le M
\left(\sum_{n=1}^{m}\Vert x_{n}\Vert^{p}\right)^{1/p}\end{equation}
(resp.,
\begin{equation}\label{convavcond}
\left(\sum_{n=1}^{m}\Vert x_{n}\Vert^{q}\right)^{1/q} \le M \left\Vert\left(\sum_{n=1}^{m}\vert x_{n}\vert^{q}\right)^{1/q}\right\Vert.)
\end{equation}
The general procedure to define the element $(\sum_{n=1}^{m}\vert x_{n}\vert^{p})^{1/p}\in\LL$ is described in \cite{LinTza1977II}*{pp. 40-41}. However, when our lattice $\LL$ is a sequence space, so that $x_{n}=(a_{j,n})_{j\in J}$ for $n=1,\dots, m$, we have
\[
\left(\sum_{n=1}^{m}\vert x_{n}\vert^{p}\right)^{1/p} = \left(\left( \sum_{n=1}^m |a_{j,n}|^p\right)^{1/p}\right)_{j\in J},
\]
and so \eqref{latticedef} and \eqref{convavcond} take a more workable form.

Every semi-normalized unconditional basis $\BB$ of a quasi-Banach space $\XX$ becomes normalized and $1$-unconditional after a suitable renorming of $\XX$, so that we can associate a sequence space to $\BB$. An unconditional basis $\BB$ will be said to be $p$-convex if its associated sequence space is. In general, we say that a semi-normalized unconditional basis has a property about lattices if its associated sequence space has it. And the other way around, i.e., we will say that a sequence space enjoys a certain property relevant to bases if its unit vector system does. A family of semi-normalized unconditional bases will be said to have a certain property if every basis in the family has it and all the constants involved (including the semi-normalization and the unconditionality ones) are uniformly bounded.

If a quasi-Banach lattice is locally convex as a quasi-Banach space, then it is $1$-convex as a quasi-Banach lattice. However, despite the fact that every quasi-Banach space is $p$-convex for some $0<p\le 1$, there exist quasi-Banach \emph{lattices} that are not $p$-convex for any $p$. Kalton introduced in \cite{Kalton1984} the concept of \emph{$L$-convex} lattice and showed that a quasi-Banach lattice is $L$-convex if and only if it is $p$-convex for some $p>0$.

\begin{lemma}\label{lem:lpconvexity} Let $A$ be a finite subset of $(0,\infty]$. Then the sequence space $\bigoplus_{p\in A} \ell_{p}$ is $q$-convex and $r$-concave, where $q=\min A$ and $r=\max A$.
\end{lemma}

\begin{proof} Minkowski's inequality yields that for any measure $\mu$, the lattice $L_p(\mu)$ is $r$-convex and $q$-concave for $0<q\le p \le r\le\infty$. Now we just need to use that $q$-convexity and $r$-concavity are preserved under direct sums.
\end{proof}

Recall that two quasi-Banach spaces $\XX_1$ and $\XX_2$ are \emph{totally incomparable} if they have no infinite-dimensional subspaces in common (up to isomorphism). It is known \cites{Stiles1972,LinTza1977} that $\ell_p$ and $\ell_q$ are totally incomparable spaces if $0<p< q\le\infty$. Let us see what happens with the direct sums of $\ell_p$ spaces.

\begin{proposition}\label{prop:B} Suppose that $\bigoplus_{p\in A} \ell_p \approx \bigoplus_{p\in B} \ell_p$ for some finite subsets $A,B$ of $(0,\infty]$.
Then $A=B$.
\end{proposition}

\begin{proof}Let $p\in A$. Applying \cite{Ortynski1981}*{Proposition 1.3} yields $(\XX_q)_{q\in B}$ such that $\XX_q$ is a complemented subspace of $\ell_q$ for every $q\in B$ and $\ell_p\approx \bigoplus_{q\in B} \XX_q$. Pick $q\in B$ such that $\XX_q$ is infinite-dimensional. Then $\ell_q$ and $\ell_p$ are not totally incomparable, and so $p=q$. We have proved that $A\subseteq B$. Switching the roles of $A$ and $B$ we obtain $B\subseteq A$.
\end{proof}

We will take advantage of the lattice structure to delve a bit deeper into the concept of totally incomparable quasi-Banach spaces.

\begin{lemma}\label{lem:TI} Let $\LL_1$ and $\LL_2$ be sequence spaces. Suppose that $\LL_1$ is $r$-convex, that $\LL_2$ is $q$-convex and that $q<r$. Then $\LL_1$ and $\LL_2$ are totally incomparable.
\end{lemma}

To prove this we need an auxiliary result.
\begin{lemma}\label{lem:net} Let $\XX$ be an infinite dimensional quasi-Banach space. For any $0<\delta<1$ there is $(f_n)_{n=1}^\infty$ in the open unit ball $B_{\XX}$ of $\XX$ with $\Vert f_n-f_m\Vert>\delta$ for all $n\not=m$ in $\NN$.
\end{lemma}

\begin{proof}We construct the sequence $(f_n)_{n=1}^\infty$ recursively. Fix $0<\delta<1$ and assume that for a given $k\in\NN$ we have chosen vectors $(f_n)_{n=1}^{k-1}$ in $B_{\XX}$ fulfilling our claim. Set $\YY_n=[\xx_n \colon 1 \le n \le k-1]$ and $\XX_n=\XX/\YY_{n-1}$. Since $\XX_n\not=\{0\}$, there is $f\in\XX$ such that $\Vert -f+\YY_{n-1}\Vert<1$ and $\Vert -f+\YY_{n-1}\Vert>\delta$. The mere definition of the norm in $\XX/\YY_{n-1}$ yields $f_n\in \XX$ such that $f+f_n\in\YY_{n-1}$ and $\Vert f_n\Vert < 1$. Let $k\le n-1$. Since $f_{k}-(f+f_n)\in\YY_{n-1}$,
\[ \Vert f_k-f_n\Vert =\Vert f+f_k-f_n-f\Vert \ge \Vert f+\YY_{n-1}\Vert>\delta.\qedhere\]
\end{proof}

\begin{proof}[Proof of Lemma~\ref{lem:TI}]Let $\YY$ be a quasi-Banach space isomorphic to a subspace of both $\LL_1$ and $\LL_2$, and let $T_i\colon\YY\to \LL_i$ be an isomorphic embedding, $i=1$, $2$. By Lemma~\ref{lem:net} we can pick a sequence $(f_n)_{n=1}^\infty$ in $B_{\YY}$ such that $\inf_{n\not=m}\Vert f_ n-f_m\Vert>0$. Using a Cantor diagonal argument, passing to a subsequence we obtain that $(T_i(f_n))_{n=1}^\infty$ is pointwise convergent in $\LL_i$, $i=1$, $2$. Then, if $\yy_n= f_{2n-1}-f_{2n}$, $(T_i(\yy_n))_{n=1}^\infty$ is pointwise convergent to zero in $\LL_i$, $i=1$, $2$. Since $\inf_n \Vert \yy_n\Vert>0$ and $\sup_n \Vert \yy_n\Vert<\infty$, applying a standard ``gliding hump'' argument (see \cite{BePe1958}) we may pass to a further subsequence and obtain that $(\yy_n)_{n=1}^\infty$ is equivalent to a block basic sequence with respect to the canonical basis of both $\LL_1$ and $\LL_2$. Then for any $(a_n)_{n=1}^\infty\in c_{00}$ we have
\[
\left(\sum_{n=1}^\infty |a_n|^q\right)^{1/q}
\lesssim\left\Vert \sum_{n=1}^\infty a_n T_2(\yy_n) \right\Vert
\approx \left\Vert \sum_{n=1}^\infty a_n T_1(\yy_n) \right\Vert
\lesssim\left(\sum_{n=1}^\infty |a_n|^r\right)^{1/r}.
\]
This absurdity proves that $\LL_1$ and $\LL_2$ are totally incomparable.
\end{proof}

A quasi-Banach space (respectively, a quasi-Banach lattice) $\XX$ is said to be \emph{sufficiently Euclidean} if $\ell_2$ is crudely finitely representable in $\XX$ as a complemented subspace (respectively, complemented sublattice), i.e., there is a positive constant $C$ such that for every $n\in\NN$ there are bounded linear maps (respectively, lattice homomorphisms) $I_n \colon\ell_2^n \to \XX$ and $P_n\colon \XX \to \ell_2^n$ with $P_n\circ I_n =\Id_{\ell_2^n}$ and $ \Vert I_n\Vert \, \Vert P_n \Vert\le C$. We say that $\XX$ is \emph{anti-Euclidean} (resp. \emph{lattice anti-Euclidean}) if it is not sufficiently Euclidean.

The theory of anti-Euclidean (locally convex) spaces goes back to \cite{CK1998}. As for nonlocally spaces, we point out that if the Banach envelope of a quasi-Banach space $\XX$ is anti-Euclidean, then $\XX$ is anti-Euclidean. However, no examples are known of anti-Euclidean quasi-Banach spaces whose Banach envelope is sufficiently Euclidean. Perhaps this gap is at the root of the lack of motivation for developing a theory of sufficiently Euclidean spaces in the setting of quasi-Banach spaces. In practice, the only thing we will use in this paper is that $\ell_1$ (the Banach envelope of $\ell_{p}$ for $0<p<1$) is anti-Euclidean. The most natural way to see this relies on Grothendick's theorem. Indeed, if $\ell_1$ were sufficiently Euclidean, by \cite{LinPel1968}*{Theorem 4.1} the identity map on $\ell_2^n$ would be absolutely summing uniformly in $n$, which is absurd.

\begin{remark}\label{rmk:fromAEtoLAE}
If a quasi-Banach space $\XX$ is anti-Euclidean then every complemented subspace $\YY$ of $\XX$ is lattice anti-Euclidean with respect to any lattice estructure on $\YY$.
\end{remark}

\begin{proposition}\label{prop:AEEl1} Suppose $\XX$ is a quasi-Banach space with a semi-normalized unconditional basis $\BB$ which dominates the unit vector basis of $\ell_1$. Then the coefficient transform regarded as a map from $\XX$ into $\ell_{1}$ is an envelope map, hence $\widehat \XX\approx \ell_{1}$ is anti-Euclidean.
\end{proposition}

\begin{proof} Lemma~\ref{lem:envelope} yields that $\widehat \BB$ is a semi-normalized unconditional basis of $\widehat \XX$. By the universal property of $\widehat \XX$, the basis $\widehat \BB$ dominates the unit vector system of $\ell_1$ and, by the local convexity of $\widehat \XX$, the converse also holds. We conclude that those two bases are equivalent, therefore $\widehat \XX\approx \ell_1$.
\end{proof}

\section{Complemented sequences in quasi-Banach spaces}\label{Prepa}
\noindent
This section is geared towards Theorem~\ref{thm:keytechnique}, which tells us that, under three straightforwardly verified conditions regarding a space and a basis, the complemented unconditional basic sequences of certain quasi-Banach spaces with unconditional basis are rather easy to classify. The techniques used in the proof of this theorem are a development of the methods introduced by Casazza and Kalton in \cites{CK1998, CK1999} to investigate the problem of uniqueness of unconditional basis in a class of Banach lattices that they called \emph{anti-Euclidean}. The subtle but crucial role played by the lattice structure of the space in the proof of Theorem~\ref{thm:keytechnique} has to be seen in that it will permit to simplify the untangled way in which complemented basic sequences can be written in terms of the basis. These techniques have been extended to the nonlocally convex setting and efficiently used in the literature to establish the uniqueness of unconditional bases of complemented subspaces of \emph{infinite} products such as $\ell_{p}(\ell_{q})=(\ell_{q}\oplus \ell_{p}\oplus\dots\oplus \ell_{p}\dots)_{q}$ for $p\in (0,1]\cup \{\infty\}$ and $q\in (0,1]\cup \{2,\infty\}$ (\cites{AKL2004,AL2008, AlbiacLeranoz2010, AL2010, AL2011}). However, the spaces $\bigoplus_{p\in A}\ell_{p}$ with $2<|A|<\infty$ are not complemented in any space $\ell_{p}(\ell_{q})$, and so the known results on the subject do not shed direct information about them.

Before moving on, recall that a basic sequence $(\uu_{i})_{i\in I}$ in a quasi-Banach space $\XX$ (with nontrivial dual) is said to be \emph{complemented} if its closed linear span $\UU= [\uu_{i} \colon i\in I]$ is a complemented subspace of $\XX$, i.e., there is a bounded linear map $P\colon\XX\to\UU$ with $P|_\UU=\Id_\UU$. Thus, a basic sequence $\BB_u=(\uu_i)_{i\in I}$ is complemented in $\XX$ if and only if there is a sequence $(\uu_i^*)_{i\in I}$ in $\XX^*$ such that $\uu_i^*(\uu_k)=\delta_{i,j}$ for every $i$, $k\in I$ and the linear map $P_u\colon\XX\to \XX$ given by
\begin{equation}\label{eq:projCUBS}
P_u(f)=\sum_{i\in I} \uu_i^*(f) \, \uu_i, \quad f\in\XX,
\end{equation}
is well-defined (hence bounded by Closed Graph Theorem). We will refer to $(\uu_i^*)_{i\in I}$ as a sequence of \emph{projecting functionals} for $\BB_u$. Notice that the sequence $(\uu_i^*)_{i\in I}$ thus defined need not be unique unless, of course, $\BB_u=(\uu_i)_{i\in I}$ spans the whole space $\XX$. A complemented basic sequence $\BB_u=(\uu_i)_{i\in I}$ in $\XX$ with mutually disjoint supports with respect to a basis $\BB$ will be said to be \emph{well complemented} if we can choose a sequence of projecting functionals $\BB_{u}^*=(\uu_i^*)_{i\in I}$ with $\supp(\uu_i^*)\subseteq \supp(\uu_i)$ for $i\in I$. In this case $\BB_{u}^*$ is called a sequence of \emph{good projecting functionals} for $\BB_{u}$. For instance, since the coordinate projections of an unconditional basis $\BB$ are bounded, the subbases of $\BB$ are trivially well complemented basic sequences with respect to $\BB$.

Let us get started with a reduction lemma.

\begin{lemma}[cf.\ \cite{AKL2004}*{Lemma 3.8}]\label{lem:k2one} Let $\BB_u=(\uu_i)_{i\in I}$ be a well complemented block basic sequence with respect to an unconditional basis $\BB=(\xx_j)_{j\in J}$ of a quasi-Banach space $\XX$, and let $(\uu_i^*)_{i\in I}$ be a sequence of good projecting functionals for $\BB_u$. Suppose $\BB_v=(\vv_i)_{i\in I}$ and $(\vv_i^*)_{i\in I}$ are sequences in $\XX$ and $\XX^{\ast}$ respectively such that for some positive constant $C$ we have
\begin{enumerate}
\item[(i)] $|\xx_j^*(\vv_i)| \le C |\xx_j^*(\uu_i)|$ for all $(i,j)\in I \times J$,
\item[(ii)] $|\vv_i^*(\xx_j)| \le C |\uu_i^*(\xx_j)|$ for all $(i,j)\in I \times J$,
and
\item[(ii)] $\inf_{i\in I} |\vv_i^*(\vv_i)|>0$ for all $i\in I$.
\end{enumerate}
Then $\BB_v$ is a well complemented block basic sequence equivalent to $\BB_u$. Moreover, if $\vv_i^*(\vv_i)=1$ and $\supp(\vv_i^*)\subseteq\supp(\vv_i)$ for all $i\in I$, then $(\vv_i^*)_{i\in I}$ is a sequence of good projecting functionals for $\BB_v$.
\end{lemma}
\begin{proof} Without loss of generality, by dilation we assume that $\vv_i^*(\vv_i)=1$ for all $i\in I$. Replacing $\vv_i^*$ with $S_{\supp(\vv_i)}^*(\vv_i^*)$ we can also assume that $\supp(\vv_i^*)\subseteq \supp(\vv_i)$ for all $i\in I$. Since $\supp(\vv_i)\subseteq\supp(\uu_i)$ for every $i\in I$, there are $\beta$ and $\gamma\in\ell_\infty$ such that $S_\gamma(\uu_i)=\vv_i$ and $S_\beta^*(\uu_i^*)=\vv_i^*$ for every $i\in I$. It follows that $\BB_v$ is a block basic sequence, thus $\vv_i^*(\vv_k)=0$ whenever $i\not=k$.

Let $P_u$ be defined as in \eqref{eq:projCUBS}. The bounded linear map $T= P_u \circ S_\beta$ satisfies
\[
T(f)=\sum_{i\in I}\uu_i^*(S_\beta(f)) \, \uu_i
=\sum_{i\in I} S_\beta^*(\uu_i^*)(f) \, \uu_i
=\sum_{i\in I} \vv_i^*(f)\, \uu_i
\]
for all $f\in \XX$. In particular, $T(\vv_i)=\uu_i$ for all $i\in I$. Therefore $\BB_u$ and $\BB_v$ are equivalent.

Finally, the bounded linear map $P_v=S_\gamma\circ T$ satisfies
\[
P_v(f)=\sum_{i\in I} \vv_i^*(f)\, \vv_i, \quad f\in \XX.
\]
We deduce that $\BB_v$ is well complemented and that $(\vv_i^*)_{i\in I}$ is a sequence of (good) projecting functionals for $\BB_v$.
\end{proof}

The following definition identifies and gives relief to an unstated feature shared by some unconditional bases. Examples of such bases can be found e.g. in \cites{Kalton1977, CK1998, AL2008}, where the property naturally arises in connection with the problem of uniqueness of unconditional basis.

\begin{definition} A semi-normalized unconditional basis $\BB=(\xx_j)_{j\in J}$ of a quasi-Banach space $\XX$ is said to be \emph{universal for well complemented block basic sequences} if for every semi-normalized well complemented block basic sequence $\BB_u=(\uu_i)_{i\in I}$ of $\BB$ there is a one-to-one map $\pi\colon I\to J$ such that $\pi(i)\in\supp(\uu_i)$ for all $i\in I$, and $\BB_u$ is equivalent to the rearranged subbasis $(\xx_{\pi(i)})_{i\in I}$ of $\BB$.
\end{definition}

\begin{remark}\label{rmk:lpK1}
The fact that the canonical basis $\EE_{\ell_{p}}$ of $\ell_p$, $0<p\le \infty$, is \emph{perfectly homogeneous} (see \cite{AlbiacKalton2016}*{Chapter 9}) implies that $\EE_{\ell_{p}}$ is universal for well complemented block basic sequences.
\end{remark}

This observation in combination with our next result gives ground to the fact that for any $A\subset (0,\infty]$ finite, the basis $\bigoplus_{p\in A}\EE_{\ell_{p}}$ of $\bigoplus_{p\in A}\ell_{p}$ is universal for well complemented block basic sequences.

\begin{proposition}\label{prop:direcsumk1} Let $(\XX_{i})_{i\in F}$ be a finite collection of quasi-Banach spaces. For each $i\in F$ suppose that $\BB_i$ is a basis of $\XX_{i}$ which is universal for well complemented block basic sequences. Then the basis $\BB=\bigoplus_{i\in F} \BB_i$ of $\XX= \bigoplus_{i\in F} \XX_{i}$ is universal for well complemented block basic sequences.
\end{proposition}

\begin{proof}
Let $(\uu_j)_{j\in J}$ be a well complemented block basic sequence in $\XX$ with respect to $\BB$ with good projecting functionals $(\uu_j^*)_{j\in J}$. Write $\uu_j=(\uu_{j,i})_{i\in F}$ and $\uu_j^*=(\uu^*_{j,i})_{i\in I}$. By hypothesis we have $\supp(\uu^*_{j,i})\subseteq\supp(\uu_{j,i})$ for every $(j,i)\in J\times F$. Since $\sum_{i\in F}\uu^*_{j,i}(\uu_{j,i})=1$, there is $\iota\colon J \to F$ such that
\[
|\uu^*_{j,\iota(j)}(\uu_{j,\iota(j)})|\ge |F|^{-1},\quad j\in J.
\]
For $i\in F$ let $S_i\colon \XX_i\to \XX$ be the inclusion map. By Lemma~\ref{lem:k2one}, the family
\[
\BB'=(S_{\iota(j)} (\uu_{j,\iota(j)}))_{j\in J}
\]
is a well complemented block basic sequence in $\XX$ equivalent to $\BB$. We infer that $\BB'_i:=(\uu_{j,i})_{j\in \iota^{-1}(i)}$ is is a well complemented block basic sequence in $\XX_i$ for every $i\in F$. By hypothesis, $\BB_i'$ is permutatively equivalent to a subbasis $\BB_{i}''$ of $\BB_i$. Hence, $\BB'$ is permutatively equivalent to the subbasis $\bigoplus_{i\in F} \BB_{i}''$ of $\BB$.
\end{proof}

\begin{theorem}\label{thm:keytechnique} Let $\XX$ be a quasi-Banach space whose Banach envelope is anti-Euclidean.
Suppose $\BB$ is an unconditional basis for $\XX$ such that:
\begin{enumerate}
\item[(i)] The lattice structure induced by $\BB$ in $\XX$ is L-convex;
\item[(ii)] $\BB$ is universal for well complemented block basic sequences; and
\item[(iii)] $\BB\sim \BB^2$.
\end{enumerate}
Then every complemented unconditional basic sequence of $\XX$ is permutatively equivalent to a subbasis of $\BB$.
\end{theorem}

\begin{proof} Let $\BB_0$ be a semi-normalized complemented unconditional basic sequence in $\XX$. The unconditional basic sequence $\widehat{\BB_0}$ spans a complemented subspace of $\widehat \XX$ and hence, by Remark~\ref{rmk:fromAEtoLAE}, $\widehat{\BB_0}$ is lattice anti-Euclidean. Theorem 3.4 from \cite{AKL2004} gives that $\BB_0$ is permutatively equivalent to a well complemented block basic sequence of $\BB^s$ for some $s\in\NN$. Since, by the hypothesis (iii), $\BB^s\sim\BB$ and $\BB$ is universal for well complemented block basic sequences, we are done.
\end{proof}

This result applied to finite direct sums $\bigoplus_{p\in A}\ell_{p}$ gives the following.

\begin{theorem}\label{thm:A}Let $A$ be a finite subset of $(0,1]$. Then every complemented unconditional basic sequence of the sequence space $\bigoplus_{p\in A}\ell_{p}$ is permutatively equivalent to the unit vector basis of $\bigoplus_{p\in B}\ell_{p}$ for some $B\subseteq A$.
\end{theorem}

\begin{proof} We have that $\bigoplus_{p\in A} \ell_{p}\subseteq \ell_{1}^{|A|}\equiv \ell_1$ (as quasi-Banach lattices). Then, the result follows by combining Lemma~\ref{lem:sumSSK3}, Lemma~\ref{lem:lpconvexity}, Proposition~\ref{prop:AEEl1}, Remark~\ref{rmk:lpK1}, Proposition~\ref{prop:direcsumk1}, and Theorem~\ref{thm:keytechnique}.
\end{proof}

\section{The proof of the Main Theorem}\label{MainSec}
\noindent
With the all groundwork of the previous sections we may proceed directly to proof of our Main Theorem.

\begin{proof}[Completion of the proof of Theorem~\ref{thm:main}]
Without loss of generality we assume that the basis $\BB=(\xx_j)_{j\in J}$ is normalized. Assume also that the set of indexes $A_{\Klt}=\{p\in A \colon p\le 1\}$ is nonempty, otherwise the result would follow from the afore-mentioned result by Edelstein and Wojtaszczyk's \cite{EdWo1976}*{Theorem~4.11}.

Put $A_{\Cvx}=A\setminus A_{\Klt}$, $\XX_{\Klt}=\bigoplus_{p\in A_\Klt} \ell_p$, $\XX_{\Cvx}=\bigoplus_{p\in A_\Cvx} \ell_p$ and $\XX=\bigoplus_{p\in A} \ell_p$, so that $\XX=\XX_{\Klt}\oplus \XX_{\Cvx}$.

Since $\XX_{\Klt}$ is contained in the $|A_\Klt|$-fold product $\ell_1^{|A_\Klt|}$ we have $\widehat{\XX_{\Klt}}\approx \ell_1$ and so $\widehat{\XX}\approx \ell_1\oplus\XX_{\Cvx}$. By Edelstein-Wojtaszczyk's theorem there is a partition $(J_p)_{p\in \{1\} \cup A_{\Cvx}}$ of $J$ such that for each $p\in \{1\} \cup A_{\Cvx}$, the sequence $(J_\XX(\xx_j))_{j\in J_p}$ generates a space isomorphic to $\ell_p$. By Lemma~\ref{lem:envelopebis}, if we put $\YY_p=[\xx_j \colon j\in J_p]$ for $p\in \{1\} \cup A_{\Cvx}$, the restriction of $J_\XX$ to $\YY_p$ is an envelope map. Therefore
\[
\widehat{\YY_p} \approx [J_\XX(\xx_j) \colon j\in J_p] \approx \ell_p, \quad p\in \{1\} \cup A_{\Cvx}.
\]

Since, by Lemmas~\ref{lem:lpconvexity} and \ref{lem:TI}, the spaces $\XX_{\Klt}$ and $\XX_{\Cvx}$ are totally incomparable, there exist complemented subspaces $\YY_{\Klt}$ of $\XX_{\Klt}$ and $\YY_{\Cvx}$ of $\XX_{\Cvx}$ such that $\YY_1\simeq \YY_{\Klt}\oplus \YY_{\Cvx}$. Taking Banach envelopes, in light of Lemma~\ref{lem:envelopetris}, we obtain
\[
\ell_1\approx \widehat{\YY_{\Klt}}\oplus \widehat{\YY_{\Cvx}}= \widehat{\YY_{\Klt}}\oplus \YY_{\Cvx}.
\]
Since $\ell_1$ and $\XX_{\Cvx}$ are totally incomparable, $\YY_{\Cvx}$ must be finite dimensional. Hence $\XX_{\Klt}\oplus\YY_{\Cvx}\approx \XX_{\Klt}$. We infer that $\YY_1$ is isomorphic to a complemented subspace of $\XX_{\Klt}$. Then, by Theorem~\ref{thm:A}, there is $B\subseteq A_{\Klt}$ such that $(\xx_j)_{j\in J_1}$ is permutatively equivalent to the unit vector system of $\bigoplus_{p\in B} \ell_p$. That is, there is a partition $(J_p')_{p\in B}$ of $J_1$ such that $(\xx_j)_{j\in J_p'}$ is (permutatively) equivalent to the unit vector system of $\ell_p$ for every $p\in B$.

If $p\in A_{\Cvx}$, the dual space of $\widehat{\YY_p}$ has a finite cotype. Moreover, by \cite{Kalton1984}*{Theorem 4.2}, $\YY_p$ is an L-convex quasi-Banach lattice. Therefore, by \cite{Kalton1986}*{Theorem 3.4}, $J_\XX|_{\YY_p}$ is an isomorphism and so $(\xx_j)_{j\in J_p}$ generates a space isomorphic to $\ell_p$.

Summing up, if we put $A'=B\cup A_{\Cvx}$ and $J_p'=J_p$ for $p\in A_{\Cvx}$, we have that $(J_p')_{p\in A'}$ is a partition of $J$ such that $(\xx_j)_{j\in J_p'}$ generates a space isomorphic to $\ell_p$. Consequently, $\XX\approx \bigoplus_{p\in A'} \ell_p$. By Proposition~\ref{prop:B}, $A=A'$.
\end{proof}

As an application we obtain new additions to the list of spaces with a unique unconditional basis. Recall that a quasi-Banach space $\XX$ with an unconditional basis $\BB$ is said to have a \emph{unique unconditional basis up to a permutation} if every normalized unconditional basis of $\XX$ is permutatively equivalent to $\BB$.

\begin{theorem}Let $A$ be a finite subset of indexes of $(0,1]\cup\{2,\infty\}$. Then the quasi-Banach space $\bigoplus_{p\in A}\ell_{p}$ has a unique unconditional basis, up to permutation.
\end{theorem}

\begin{proof}It just follows by combining Theorem~\ref{thm:main} and the uniqueness of unconditional basis of $\ell_p$ for $p\in A$ (\cites{Kalton1977, KotheToeplitz1934, LinPel1968}). \end{proof}

\begin{remark} Observe that Theorem~\ref{thm:keytechnique} provides a proof of the uniqueness of unconditional basis of $\ell_p$ for $p\in(0,1]\cup\{\infty\}$.
\end{remark}




\begin{bibsection}
\begin{biblist}

\bib{AABW2019}{article}{
author={Albiac, F.},
author={Ansorena, J.~L.},
author={Bern\'a, P.},
author={Wojtaszczyk, P.},
title={Greedy approximation for biorthogonal systems in quasi-Banach spaces},
journal={Submitted},
}

\bib{AlbiacKalton2016}{book}{
author={Albiac, F.},
author={Kalton, N.~J.},
title={Topics in Banach space theory, 2nd revised and updated edition},
series={Graduate Texts in Mathematics},
volume={233},
publisher={Springer International Publishing},
date={2016},
pages={xx+508},
}

\bib{AKL2004}{article}{
author={Albiac, F.},
author={Kalton, N.},
author={Ler\'{a}noz, C.},
title={Uniqueness of the unconditional basis of $\ell_1(\ell_p)$ and $\ell_p(\ell_1)$, $0<p<1$},
journal={Positivity},
volume={8},
date={2004},
number={4},
pages={443--454},
}

\bib{AL2008}{article}{
author={Albiac, F.},
author={Ler\'{a}noz, C.},
title={Uniqueness of unconditional basis in Lorentz sequence spaces},
journal={Proc. Amer. Math. Soc.},
volume={136},
date={2008},
number={5},
pages={1643--1647},
}

\bib{AlbiacLeranoz2010}{article}{
author={Albiac, F.},
author={Ler\'{a}noz, C.},
title={An alternative approach to the uniqueness of unconditional basis of $\ell_p(c_0)$ for $0<p<1$},
journal={Expo. Math.},
volume={28},
date={2010},
number={4},
pages={379--384},
}

\bib{AL2010}{article}{
author={Albiac, F.},
author={Ler\'{a}noz, C.},
title={Uniqueness of unconditional basis in quasi-Banach spaces which are
not sufficiently Euclidean},
journal={Positivity},
volume={14},
date={2010},
number={4},
pages={579--584},
}

\bib{AL2011}{article}{
author={Albiac, F.},
author={Ler\'{a}noz, C.},
title={Uniqueness of unconditional bases in nonlocally convex $\ell_1$-products},
journal={J. Math. Anal. Appl.},
volume={374},
date={2011},
number={2},
pages={394--401},
}

\bib{Aoki1942}{article}{
author={Aoki, T.},
title={Locally bounded linear topological spaces},
journal={Proc. Imp. Acad. Tokyo},
volume={18},
date={1942},
pages={588\ndash 594},
}

\bib{BePe1958}{article}{
author={Bessaga, C.},
author={Pe{\l}czy{\'n}ski, A.},
title={On bases and unconditional convergence of series in Banach spaces},
journal={Studia Math.},
volume={17},
date={1958},
pages={151--164},
}

\bib{CK1998}{article}{
author={Casazza, P.~G.},
author={Kalton, N.~J.},
title={Uniqueness of unconditional bases in Banach spaces},
journal={Israel J. Math.},
volume={103},
date={1998},
pages={141--175},
}

\bib{CK1999}{article}{
author={Casazza, P.~G.},
author={Kalton, N.~J.},
title={Uniqueness of unconditional bases in $c_0$-products},
journal={Studia Math.},
volume={133},
date={1999},
number={3},
pages={275--294},
}

\bib{EdWo1976}{article}{
author={Edelstein, I.~S.},
author={Wojtaszczyk, P.},
title={On projections and unconditional bases in direct sums of Banach spaces},
journal={Studia Math.},
volume={56},
date={1976},
number={3},
pages={263--276},
}

\bib{KaPe1961}{article}{
author={Kadec, M.~I.},
author={Pe\l czy\'nski, A.},
title={Bases, lacunary sequences and complemented subspaces in the spaces $L_{p}$},
journal={Studia Math.},
volume={21},
date={1961/1962},
pages={161--176},
}

\bib{Kalton1977}{article}{
author={Kalton, N.~J.},
title={Orlicz sequence spaces without local convexity},
journal={Math. Proc. Cambridge Philos. Soc.},
volume={81},
date={1977},
number={2},
pages={253--277},
}

\bib{Kalton1984}{article}{
author={Kalton, N.~J.},
title={Convexity conditions for nonlocally convex lattices},
journal={Glasgow Math. J.},
volume={25},
date={1984},
number={2},
pages={141--152},
}

\bib{Kalton1986}{article}{
author={Kalton, N. J.},
title={Banach envelopes of nonlocally convex spaces},
journal={Canad. J. Math.},
volume={38},
date={1986},
number={1},
pages={65--86},
}

\bib{KPR1984}{book}{
author={Kalton, N.~J.},
author={Peck, N.~T.},
author={Roberts, J.~W.},
title={An $F$-space sampler},
series={London Mathematical Society Lecture Note Series},
volume={89},
publisher={Cambridge University Press},
place={Cambridge},
date={1984},
}

\bib{KotheToeplitz1934}{article}{
author={K\"othe, G.},
author={Toeplitz, O.},
title={Lineare Raume mit unendlich vielen Koordinaten und Ringen unendlicher Matrizen},
journal={J. Reine Angew Math.},
volume={171},
date={1934},
pages={193--226},
}

\bib{LinPel1968}{article}{
author={Lindenstrauss, J.},
author={Pe\l czy\'{n}ski, A.},
title={Absolutely summing operators in $L_{p}$-spaces and their applications},
journal={Studia Math.},
volume={29},
date={1968},
pages={275--326},
}

\bib{LinTza1977}{book}{
author={Lindenstrauss, J.},
author={Tzafriri, L.},
title={Classical Banach spaces. I},
note={Sequence spaces; Ergebnisse der Mathematik und ihrer Grenzgebiete, Vol. 92},
publisher={Springer-Verlag, Berlin-New York},
date={1977},
pages={xiii+188},
}

\bib{LinTza1977II}{book}{
author={Lindenstrauss, J.},
author={Tzafriri, L.},
title={Classical Banach spaces. II},
note={Function spaces;
Ergebnisse der Mathematik und ihrer Grenzgebiete, Vol. 92},
publisher={Springer-Verlag, Berlin-New York},
date={1977},
}

\bib{Ortynski1981}{article}{
author={Ortynski, A.},
title={Unconditional bases in $\ell_{p}\bigoplus\ell_{q},$ $0<p<q<1$},
journal={Math. Nachr.},
volume={103},
date={1981},
pages={109--116},
}

\bib{Rolewicz1957}{article}{
author={Rolewicz, S.},
title={On a certain class of linear metric spaces},
language={English, with Russian summary},
journal={Bull. Acad. Polon. Sci. Cl. III.},
volume={5},
date={1957},
pages={471\ndash 473, XL},
}

\bib{Singer1962}{article}{
author={Singer, I.},
title={Some characterizations of symmetric bases in Banach spaces},
journal={Bull. Acad. Polon. Sci. S\'er. Sci. Math. Astronom. Phys.},
volume={10},
date={1962},
pages={185--192},
}

\bib{Stiles1972}{article}{
author={Stiles, W.~J.},
title={Some properties of $\ell_{p}$, $0<p<1$},
journal={Studia Math.},
volume={42},
date={1972},
pages={109--119},
}

\end{biblist}
\end{bibsection}
\end{document}